\documentclass[11pt]{article}
\usepackage{amsthm}

\usepackage[utf8]{inputenc}
\usepackage{amsmath}
\usepackage{amsfonts}
\usepackage{amssymb}
\usepackage{graphicx}
\usepackage{mathabx}
\usepackage{bbm}
 \usepackage{pdfsync}
\usepackage[english]{babel}
\usepackage[colorlinks]{hyperref}

\usepackage{makeidx}         
\usepackage{multicol}        
\usepackage[bottom]{footmisc}
\usepackage{mathabx}

\DeclareMathOperator{\NN}{\mathbb{N}}

\DeclareMathOperator{\RR}{\mathbb{R}}

\DeclareMathOperator{\PP}{\mathbb{P}}

\newtheorem{theorem}{Theorem}[section]
\newtheorem{lemma}{Lemma}[section]
\newtheorem{corollary}{Corollary}[section]
\newtheorem{proposition}{Proposition}[section]
\newtheorem{remark}{Remark}[section]

\author{Yves Le Jan}
\title{Sampling biased spanning forests through marked vertices}
\begin{document}
\maketitle
\begin{abstract}
We study biased spanning forests on the complete graph \(K_N\), obtained by
assigning weight proportional to \(\kappa^q\) to spanning trees on
\(K_N\cup\{\Delta\}\), where \(q\) is the degree of a distinguished root
\(\Delta\).

Fixing a finite set \(L\) of marked vertices, we analyze the minimal
\(\Delta\)-rooted subtree connecting \(L\) to \(\Delta\). In this sense, we
investigate the effect of partially sampling a large random spanning forest
through finitely many vertices.

For fixed \(\kappa\), the reduced subtree is asymptotically a uniformly
distributed binary tree, and graph distances rescaled by \(\sqrt N\) converge
jointly in distribution to the explicit limit introduced by Aldous in his
study of the Brownian CRT. We show that the scale
\(\kappa\asymp\sqrt N\) is critical: if \(\kappa=o(\sqrt N)\), the marked
vertices lie asymptotically in a single component; if
\(\kappa\gg\sqrt N\), the induced partition is asymptotically discrete and
the distances to \(\Delta\) are negligible on the \(\sqrt N\) scale.

In the critical regime \(\kappa=c\sqrt N\), the induced partition of the
marked vertices converges to a non-degenerate limit law, namely the
\(\tfrac12\)-stable Poisson--Kingman partition studied by Pitman. We also
describe a continuous-time edge-cutting dynamics: under the critical scaling,
its fixed-time marginals are obtained by the parameter shift
\(c\mapsto c+a\).

Our approach remains entirely discrete and combinatorial: the asymptotics are
obtained from explicit determinant formulas and elementary expansions, without
invoking continuum limits.

\end{abstract}
\section{Introduction}

Given a random spanning tree $\Upsilon$ of the complete graph
$K_N$ with $N$ vertices, and a fixed subset
$L=\{a_1,\dots,a_l\}$ of $l<N$ labelled vertices,
a natural question is the following:
how are these $l$ vertices connected within the tree when $N$ is large?
More precisely, considering the smallest subtree
$\Upsilon_L$ of $\Upsilon$ containing $L$,
what is the asymptotic distribution, as $N\to\infty$,
of its topological type and of the graph distances
between its leaves and internal nodes?

\medskip
\paragraph{Spanning forests with killing.}
We formulate the problem in a slightly more general setting.
We distinguish one vertex $\Delta$ as a root and consider
random spanning trees on $K_N\cup\{\Delta\}$
with weight proportional to $\kappa^{q}$,
where $\kappa>0$ is a parameter and $q$
denotes the degree of $\Delta$.
Equivalently, when restricted to $K_N$,
this induces a $\kappa$-biased spanning forest
in which a forest with $q$ components has weight proportional to $\kappa^{\,q-1}$.

When $\kappa$ is fixed, this model differs only mildly
from the uniform spanning tree.
However, when $\kappa$ varies with $N$,
qualitatively new phenomena appear.
In particular, the scale $\kappa\asymp\sqrt N$
turns out to be critical.

\medskip
\paragraph{Main mechanism.}
The asymptotic analysis throughout the paper follows a single principle.
For fixed reduced tree shapes and fixed integer extension vectors,
the Green-determinant formula yields an explicit finite-\(N\) expression.
After a suitable $\sqrt N$ rescaling of the edge lengths,
this expression admits a limit.
Summing these limits over all integer length vectors
and all binary reduced shapes produces Riemann sums
which converge to explicit integrals.
A crucial point is that the sum of these integrals equals~$1$.
This shows that asymptotically no probability mass is lost
outside the $\sqrt N$ scaling window
and outside the class of binary reduced trees.
Binary shapes therefore emerge with total limiting mass~$1$,
and their distribution becomes uniform by symmetry.

\medskip
\paragraph{Fixed-$\kappa$ regime.}
In Section~2 we show that,
for fixed $\kappa$ and fixed $l$,
the value of $\kappa$ is asymptotically irrelevant
for the local structure of $\Upsilon_L$:
\emph{the reduced tree of $\Upsilon_L$ is asymptotically
uniformly distributed among binary trees with leaves $L$,
and the graph distances, rescaled by $\sqrt N$,
converge jointly to an explicit distribution.}
As $l$ varies, these distributions form a consistent family.

This family coincides with the consistent family of
``proper $k$-trees'' introduced by Aldous in his study of 
the Brownian CRT (continuum random tree).
In contrast, the present paper derives these asymptotic laws
purely by discrete enumeration and elementary asymptotics,
without appealing to continuum constructions.

\medskip

 \paragraph{Critical regime and phase transition.}
A different phenomenon occurs when
\[
\kappa = c\sqrt N, \qquad c>0.
\]

In Section~3 we show that in this regime the reduced object $Q_L$
obtained from $\Upsilon_L$ is, with probability tending to one, a \emph{bouquet}
consisting of $r\in\{1,\dots,l\}$ binary trees attached to $\Delta$.
Moreover, for each fixed $r$ and each fixed bouquet configuration
(i.e. a partition $\pi$ of $L$ into $r$ blocks, together with a binary tree on each block),
\emph{all such configurations have the same limiting probability}, namely
\[
\PP\bigl(Q_L=\text{given bouquet with $r$ trees}\bigr)
\longrightarrow
I_{l,r}(c),
\]
where $I_{l,r}(c)$ admits an explicit one-dimensional integral representation on $\RR_+$.
In particular,
\[
\PP\bigl(|\Pi_N|=r\bigr)\longrightarrow C_{l,r}\,I_{l,r}(c),
\]
where $|\Pi_N|$ is the number of blocks of the induced partition $\Pi_N$ of $L$ and
$C_{l,r}$ is the number of binary forests on $l$ labelled leaves with $r$ components.
Conditionally on $|\Pi_N|=r$, 
the internal subtrees defined by the blocks remain asymptotically independent,
following the same binary-tree limit law as in the fixed-$\kappa$ case.
\paragraph{}
The limiting partition of \(L\) obtained in the regime
\(\kappa\sim c\sqrt N\) is exactly the \(\tfrac12\)-stable
Poisson--Kingman partition of Pitman~\cite{Pitman2003PK}.

 In Gnedin and Pitman terminology~\cite{GnedinPitman2006},
it is the Gibbs partition with block weights $w_m=(2m-3)!!$
and explicitly computable coefficients $I_{l,r}(c)$. The coefficients satisfy the normalization $\sum_{r=1}^l C_{l,r}\,I_{l,r}(c)=1$.\\

We show in section 4 that this scale $\kappa\sim c\sqrt N$
 is critical.
If $\kappa=o(\sqrt N)$, the induced partition of $L$
is asymptotically trivial: with high probability
all marked vertices lie in a single component.
If $\kappa\gg\sqrt N$, the opposite triviality holds:
the induced partition is asymptotically the discrete partition,
and the graph distance from each marked vertex to $\Delta$
is $o(\sqrt N)$.

\noindent

We also describe in Section~5 a continuous-time edge-cutting dynamics. Under
the critical scaling, its fixed-time marginals are obtained by the simple
parameter shift \(c\mapsto c+a\). At the metric-bouquet level, this gives a
natural fragmentation/coalescence picture.
 Finally, note that a multi-cemetery extension could be treated similarly. Wilson's
algorithm can be run in such a way that whenever a loop-erased walk is killed,
one records which cemetery point was used. This produces a spanning forest
whose components are colored by cemetery labels. Here we focus on the
single-root case, since multitype versions of the construction and of the
asymptotic analysis require no new ideas.

The relation with previous work is discussed in the final section.

\section{The fixed-\(\kappa\) regime}
\paragraph{Green determinants.}
 Consider a graph $\mathcal G$ equipped with conductances  and a non-vanishing killing function defined on the set of vertices, here denoted by $X$. These data allow to define a Green matrix $G $ indexed by  pairs of vertices. Adding a cemetery point   $\Delta$, we can consider the value of the killing function at any vertex $x$ as a conductance between $x$ and $\Delta$. Then a spanning forest of rooted trees on $\mathcal G$ can be identified to a spanning tree on the extended graph $\mathcal G_{\Delta}$, rooted at $\Delta$.   An extension of Cayley's theorem (cf. for example \cite{stfl}, section 8-2 in \cite{book}) known as the (weighted) matrix-tree theorem shows that if we define the weight of such a spanning tree to be the product of the conductances of its edges, the sum of these weights is the inverse of the determinant of $G$. It provides  naturally a probability $\mathbb P$ on spanning trees rooted at $\Delta$ (and consequently a probability on spanning forests.) \\
  This applies in particular to the complete graph $K_N$ with vertices $\{1,2,...,N\}$ endowed with unit conductances and a constant killing function $\kappa$.
 The corresponding probability 
  on $\Delta$-rooted spanning trees will be denoted by $\mathbb P^{(N),\kappa}$. If $\kappa=1$, the spanning tree rooted at $\Delta$ can be identified with a uniform spanning tree on $K_{N+1}$. Local limits for these objects have been determined in \cite{dnm}, extending a result of Grimmett \cite{Grim}. \\

 The expression of the Green matrix $G^{(N),\kappa}$ is:
\[
\frac{1}{\kappa (N+\kappa)}(\kappa I+J)
\]
where $J$ denotes the $(N,N)$ matrix with all entries equal to $1$. 
It is easy to check that for any $(d,d)$  square matrix $M$ with diagonal entries equal to $a+b$ and off diagonal entries equal to $b$, $\det(M)=a^{d-1}(a+db)$. 
Hence we get the following identity:
 \begin{equation}
\det(G^{(N),\kappa})=\frac{1}{\kappa (N+\kappa)^{N-1}}
 \end{equation}
 
 Moreover, for $0<d<N$,
  \begin{equation}\label{lemdet}
\det(G^{(N),\kappa}_{i,j},\, 1\leq i,j\leq d)= \frac{d+\kappa}{\kappa(N+\kappa)^d}
 \end{equation}

 \paragraph{Wilson's algorithm.} Coming back to the general case,  a graph equipped with conductances and a killing function is naturally associated with a random walk. In the complete graph case, at each step, the  random walk is killed with probability $\frac{\kappa}{\kappa+N-1}$ and otherwise jumps to a randomly chosen different vertex.\\  A simple way to sample a spanning tree $\Upsilon$ under the probability $\mathbb P$ is to perform Wilson's algorithm (cf. \cite{wil}) which is based on loop erasure. The probability of a loop-erased path $\xi$ from a given vertex $\xi_0$ to $\Delta$ is the product of its conductances multiplied by the determinant of the restriction of the Green matrix to the vertices of $\xi$. Wilson's algorithm starts by choosing any order on the vertices and constructing a loop erased path from the first vertex to $\Delta$. Then at each step it constructs a loop erased path from the first unused vertex to the tree of used vertices (including $\Delta$) until all vertices have been used.\\We are interested in the smallest subtree $\Upsilon_L$ of $\Upsilon$ connecting a set $L$ of $l$ vertices and $\Delta$. It is a tree embedded in the graph, rooted at $\Delta$ containing $L$ and whose set of leaves, denoted by $L^*$, is included in $L$. We say that such a tree  is a $ L$-tree so that  $\Upsilon_L$ is always a $ L$-tree.  To sample $\Upsilon_L$, we can run Wilson's algorithm after choosing an order in which the vertices of $L$ are the first $l$ vertices. Then, \emph{the spanning tree  produced by the algorithm once all vertices of $L$ have been covered is exactly $\Upsilon_L$.}  This happens in at most $l$ steps.\\
  
  It is easy to check from the proof of Wilson’s algorithm given in \cite{book} that for any tree $Y$ rooted at $\Delta$,  containing $L$ and whose set of leaves is included in $L$, the probability $\mathbb P(Y= \Upsilon_L)$ is given by the product of the conductances of the edges of $Y$ multiplied by the determinant of the restriction of the Green matrix to the vertices of $Y\setminus \{\Delta\}$.\\
Consequently, in the case of the complete graph with $N$ vertices $\{1,...,N\}$ endowed with unit conductances and a constant killing factor $\kappa$, if $Y$ has $d$ vertices (root excepted) and $r$ edges incident to $\Delta$, by formula \ref{lemdet}, we have: 
 
 \begin{lemma}\label{lemdet2}
$\mathbb{P}^{(N),\kappa}(\Upsilon_L=Y)=\frac {\kappa^{r-1}(d+\kappa)}{(N+\kappa)^d} $
\end{lemma}
\paragraph{Reduced trees and lengths.}
A  spanning tree defines a map $p$ from $X$ into $X\cup \{\Delta\}$ fixing $\Delta$ and such that for any vertex $x$ in $X$ and positive integer $m$, $p^m(x)\neq x$. If $p(y)=x$, we say that $y$ is a child of $x$. If  $p^m(y)=x$ for some non-negative (positive) integer $m$, we say that $y$ is above (strictly above) $x$ and that $x$ is below (strictly below) $y$. A  node is by definition a vertex with at least two children. These notions are extended to any tree embedded in the graph and rooted at $\Delta$ when we replace $X$ by the set of tree vertices. \\
We denote by $\mathcal N$ the set of nodes of $\Upsilon _L$. There is a mapping $j$ from $\mathcal N\cup L$ into parts of $ L$. It maps a vertex $x$ to the set of elements of $L$  which are above it in the tree. In particular, it sends a leaf $i$ into $\{i\}$. We also set  $j(\Delta)=L$.\\
 We say that a $L$ tree is reduced if all its vertices which are not elements of $L$ or nodes are removed. A $L$-tree can be decomposed into its $L$-reduction (which, in the case of the complete graph, is still a $L$ -tree) and finite sequences of intermediate internal vertices with a single child. \\
 
Let us now consider the case of the complete graph. It is clearly enough to consider the case $L=\{1,...l\}$. Then the natural  order on $L$ is used to  define an order on the children of each internal node: we order these children by the minimum element of the set of leaves above them.
 (In the binary case this simply fixes a canonical
left/right orientation at each branch point.)
This turns $Y$ into a plane rooted tree in a deterministic way.\\
 Perform the depth-first contour traversal of the
resulting plane rooted tree, starting and ending at $\Delta$; this contour visits each
edge exactly twice and defines an order on all vertices of $Y$.


In order to capture the way elements of $L$ are connected, in other words the geometry of the subtree connecting them, only the labels in $L$ do matter. This leads us to define an equivalence relation: we say that two $L$-trees are equivalent if they can be transformed into one another by a permutation of $ \{l+1,...,N\}$. Note that equivalence preserves the order we defined on $\mathcal N\cup L$. The (equivalence) class of the $L$-reduction of $\Upsilon_L$ denoted by $Q_L$ inherits a tree structure. Its internal vertices are in one to one correspondence with $\mathcal N\cup L\setminus L^*$ and its leaves with $L^*$. The class of \(\Upsilon_L\) is determined by \(Q_L\) together with an extension vector indexed by the vertices
\(\mathcal N\cup L\) of the reduced tree other than \(\Delta\). For
\(x\in\mathcal N\cup L\), let \(e_x\) be the edge of the reduced tree ending
at \(x\), oriented away from \(\Delta\). We define
$u_L(x)$ as the number of unlabelled vertices of \(\Upsilon_L\) lying strictly between the
endpoints of the segment corresponding to \(e_x\) plus one added if  $x \notin L$.
 Thus, with this convention,
\[
\sum_{e\in E(Q_L)}u_L(e)
=
|\mathrm{Vert}(\Upsilon_L)\setminus(L\cup\{\Delta\})|.\]
We can say that $(Q_L,u_L)$ is a $\mathbb N$-extension of $Q_L$.\\
 
  We say that a $L$-tree is binary iff $L^*=L$, every internal vertex has exactly two children and the root only one. An easy induction shows that the set $\mathcal  B_l$ of binary $L$-tree classes has cardinality $c_l=\prod_{i=1}^{l-1} (2i-1)= \frac{(2l-3)!}{2^{l-2}(l-2)!}$ and that these trees have $l-1$ internal vertices. Indeed, adding one leaf is done by choosing a vertex, adding  a node just below it and connecting it to the new leaf. Consider now binary tree extensions: using the contour order on leaves and nodes, $u_L$  becomes a $(2l-1)$-tuple of non-negative integers $(u_L(i), 1\leq i \leq 2l-1)$, which are positive except possibly for those corresponding to leaves. $\sum_i u_L(i) $ is the number of inner vertices of $\Upsilon _L$. \\ 
 
\paragraph {Asymptotics.}We can now formulate precisely our problem which is to show that as $N$ increases to infinity, the probability that the reduction  of $\Upsilon_L$ is binary converges to one, and to determine the asymptotic behavior of the joint distribution of the pair $(Q_L,u_L)$. This is given in the following:

\begin{theorem}\label{blu}
a) $ \mathbb P^{(N),\kappa}( L^*=L)$ converges to 1 as $N\uparrow \infty$.\\
b) Given $\alpha \in \mathcal B_l$, $$\lim_{N\uparrow \infty} \mathbb P ^{(N),\kappa} (Q_L=\alpha)=c_l^{-1}.$$
 i.e. the distribution of $Q_L$ converges towards the uniform distribution on $\mathcal B_l$.\\
c) Given $t\in  (0,\infty)^{2l-1}$, $$\lim_{N\uparrow \infty}N^{(2l-1)/2} \mathbb P^{(N),\kappa} (Q_L=\alpha, u_L=[t\sqrt N])=(\sum_1^{2l-1}t_i)e^{-(\sum_1^{2l-1}t_i)^2/2}$$ in which $[t\sqrt N]$ denotes the $2l-1$-tuple of integers
 $([\sqrt N t_1],...,[\sqrt N t_{2l-1}])$.\\
\end{theorem}

\begin{proof}
Note that a) follows directly from b) since there are $c_l$ binary trees with $l$ labeled leaves.\\
Then we start with the proof of c). Note that setting $\Sigma=\sum_1^{2l-1}[t_i\sqrt N]$
 there are $\prod_{i=0}^{\Sigma-1}(N-l-i)$ possible choices for the internal vertices of $\Upsilon_L$, given that $u_L=[t\sqrt N]$.
From lemma \ref{lemdet2}, we get that $$\mathbb P^{(N),\kappa} (Q_L=\alpha, u_L=[t\sqrt N])=\frac{\Sigma+l+\kappa}{(N+\kappa)^{\Sigma+l}}\prod_{i=0}^{\Sigma -1}(N-l-i)$$ (with the usual convention that an empty product equals 1).

We define
\[
L_N:=\log(\frac{\Sigma+l+\kappa}{(N+\kappa)^{\Sigma+l}}\prod_{i=0}^{\Sigma-1}(N-l-i)),\]

\begin{lemma}
Set $\Sigma=\sigma\sqrt N$, and assume that
$0\le \sigma\le \sigma_{\max}$, with \(l\) and \(\kappa\) fixed.  
There exists a constant \(C>0\)  (depending on
\(\sigma_{\max},l,\kappa\)) such that for all 
\(\sigma\in[0,\sigma_{\max}]\), and $N$ large enough:
\[
\left|\,L_N-\Bigl(\log(\Sigma+l+\kappa)-l\log N-\frac{\sigma^2}{2}\Bigr)\right|
\le \frac{C}{\sqrt N}.
\]
\end{lemma}
\begin{proof}

\[
L_N=\log(\Sigma+l+\kappa)-(\Sigma+l)\log(N+\kappa)
+\sum_{i=0}^{\Sigma-1}\log(N-l-i),
\]
with $\Sigma=\sum_{i=1}^{2l-1}k_i$ and
$\Sigma\le \sigma_{\max}\sqrt N$.

\medskip
\noindent
Rewrite
\[
\log(N+\kappa)=\log N+\log\!\left(1+\frac{\kappa}{N}\right),
\qquad
\log(N-l-i)=\log N+\log\!\left(1-\frac{l+i}{N}\right),
\]
so that
\[
L_N=\log(\Sigma+l+\kappa)-l\log N
-(\Sigma+l)\log\!\left(1+\frac{\kappa}{N}\right)
+\sum_{i=0}^{\Sigma-1}\log\!\left(1-\frac{l+i}{N}\right).
\]
First

\[
\left|(\Sigma+l)\log\!\left(1+\frac{\kappa}{N}\right)\right|
\le
(\Sigma+l)\frac{\kappa}{N}\le
\frac{\kappa(\sigma_{\max}+l)}{\sqrt N}
:=
\frac{C_1}{\sqrt N}.
\]

\medskip
\noindent
Note that on $(0, 1/2]$,
\[
\log(1-x)=-x-\varepsilon(x)
\quad \text{with}\quad
0\le \varepsilon(x) \le x^2
\]
and set $x_i=(l+i)/N$. Assuming $N$ large enough so that $l+\sigma_{max}\sqrt N\le N/2$,
\[
\sum_{i=0}^{\Sigma-1}\log\!\left(1-\frac{l+i}{N}\right)
=
-\frac1N\sum_{i=0}^{\Sigma-1}(l+i)
-\sum_{i=0}^{\Sigma-1}\varepsilon(x_i).
\]
\medskip
\noindent
We have
\[
0\le \sum_{i=0}^{\Sigma-1}\varepsilon(x_i)
\le
\sum_{i=0}^{\Sigma-1}\left(\frac{l+i}{N}\right)^2
\le
\frac{1}{ N^2}\,\Sigma\,(l+\Sigma)^2
\le
\frac{C_2}{\sqrt N},
\]
with $C_2=\sigma_{\max}(l+\sigma_{\max})^2$.\\

\noindent
Moreover
\[
\frac1N\sum_{i=0}^{\Sigma-1}(l+i)
=
\frac1N\left[\Sigma l+\frac{\Sigma(\Sigma-1)}2\right]=
\frac{\Sigma^2}{2N}
+\frac{\Sigma l}{N}
-\frac{\Sigma}{2N}.
\]
\medskip
\noindent
Hence
\[
\left|
\frac1N\sum_{i=0}^{\Sigma-1}(l+i)
-
\left(\frac{\Sigma^2}{2N}\right)
\right|
\leq
\frac{\Sigma}{2N}+\frac{\Sigma l}{N}
\le
\frac{\sigma_{\max}+2l}{2\sqrt N}
=: \frac{C_3}{\sqrt N}.
\]

\medskip
\noindent
Therefore defining
\[
r_N:=L_N-
\left(\log(\Sigma+l+\kappa)-l\log N-\frac{\Sigma^2}{2N}\right)\]

\[ \left| r_N \right|
\le
\frac{C}{\sqrt N}
\]
with $C=C_1+C_2+C_3$,
uniformly for $\Sigma\le\sigma_{\max}\sqrt N$ and $N$ large enough.
\end{proof}

\medskip
\noindent
Then
\[
e^{L_N}=(\Sigma+l+\kappa)\,N^{-l}e^{-\Sigma^2/(2N)}\,e^{r_N}.
\]
For $N$ large enough so that $C/\sqrt N\le 1$, we have $|e^{r_N}-1|\le 2|r_N|$, hence
\[
\left|e^{L_N}-(\Sigma+l+\kappa)\,N^{-l}e^{-\Sigma^2/(2N)}\right|
\le \frac{2C}{\sqrt N}\,(\Sigma+l+\kappa)\,N^{-l}e^{-\Sigma^2/(2N)},
\]
uniformly for all $\sigma\in[0,\sigma_{\max}]$.
 Moreover, 
 for all $N$ sufficiently large and all admissible
$k\in\mathbb N^{2l-1}$ satisfying $\sum k_i\le \sigma_{\max}\sqrt N$ ( and $k_i >0$ on internal nodes),
\[
\left|
N^{(2l-1)/2}
\mathbb P^{(N),\kappa}(Q_L=\alpha,u_L=k)
-
\left(\sum_{i=1}^{2l-1}\frac{k_i}{\sqrt N}\right)
\exp\!\left(-\frac{(\sum_{i=1}^{2l-1}k_i)^2}{2N}\right)
\right|
\le \frac{C'}{\sqrt N}.
\]
with $C'\le (2C\sigma_{max}+1)$ This concludes the proof of c).\\
\medskip
\noindent
To prove b), we first compute the normalizing integral:
\[
\int_{\mathbb R_+^{2l-1}}
\left(\sum_{i=1}^{2l-1} t_i\right)
\exp\!\left(-\frac{(\sum_{i=1}^{2l-1} t_i)^2}{2}\right)
dt
=
c_l^{-1}.
\]
\medskip
\noindent
Fix $\sigma_{\max}>0$. By the uniform estimate obtained in part (c),
\begin{align*}
\liminf_{N\to\infty}
\mathbb P^{(N),\kappa}(Q_L=\alpha)
&\ge
\liminf_{N\to\infty}
\sum_{\substack{k\in\mathbb N^{2l-1}\\ \sum k_i\le \sigma_{\max}\sqrt N}}
\mathbb P^{(N),\kappa}(Q_L=\alpha,u_L=k)
\\
&=
\liminf_{N\to\infty}
N^{-(2l-1)/2}
\sum_{\substack{k\in\mathbb N^{2l-1}\\ \sum k_i\le \sigma_{\max}\sqrt N}}
\left(\sum_{i=1}^{2l-1}\frac{k_i}{\sqrt N}\right)
\exp\!\left(-\frac{(\sum_{i=1}^{2l-1}k_i)^2}{2N}\right).
\end{align*}
\medskip
\noindent
The last sum is the Riemann sum, with mesh size \(h_N=1/\sqrt N\), over the
simplex
\[
\Delta_{\sigma_{\max}}
:=
\left\{
t\in\mathbb R_+^{2l-1}:\ \sum_{i=1}^{2l-1}t_i\le \sigma_{\max}
\right\}.
\]
Therefore,
\[
\liminf_{N\to\infty}
\mathbb P^{(N),\kappa}(Q_L=\alpha)
\ge
\int_{\Delta_{\sigma_{\max}}} f(t)\,dt .
\]
\medskip
\noindent
Letting $\sigma_{\max}\to\infty$, monotone convergence gives
\[
\liminf_{N\to\infty}
\mathbb P^{(N),\kappa}(Q_L=\alpha)
\ge
\int_{\mathbb R_+^{2l-1}} f(t)\,dt
=
c_l^{-1}.
\]
\medskip
\noindent
Since there are exactly $c_l$ binary shapes, we must have equality:
\[
\lim_{N\to\infty}
\mathbb P^{(N),\kappa}(Q_L=\alpha)
=
c_l^{-1}.
\]
This completes the proof of b).
\end{proof}

 \paragraph{Additional comments}
 \medskip
\noindent
 1) It follows from a) that any finite set of vertices asymptotically belong to the same tree of the spanning forest. \\
 
\medskip
\noindent 
2) Taking $l=1$, we see that in particular, as $N$ increases to infinity, the distribution of the graph distance of any fixed vertex to the root rescaled by $\sqrt N$ converges to the density $xe^{-x^2/2}$. Note that the scaling by $\sqrt N$ appeared already in \cite{Szkr}.
\\

\medskip
\noindent
3) If $\kappa=1$, the result can be interpreted as giving the asymptotic distribution of the subtree connecting $l+1$ vertices in the unrooted uniform spanning tree on $K_{N+1}$. \\

 \medskip
\noindent
 4) Note that for finite $N$, $Q_L$ can be non-binary but the corresponding  probability tends to 0 as $\uparrow \infty$.  \\
 
\medskip
\noindent 
5)  Note that the limiting distribution we get is independent of $\kappa$.
On the other hand, it is shown in \cite{book}, among other results, that the probability  for two vertices to belong to different trees is equivalent to $\frac{\kappa\sqrt \pi}{\sqrt{2n}}$ and that the  probability  for two vertices to be on the same branch starting from the root is equivalent to $\frac{\sqrt 2 \pi}{\sqrt{n}}$. A more general result could be looked for, considering all types of atypical topologies for $Q_L$\\

\section{The critical case: $\kappa=c\sqrt N$}

We now consider the case where the killing parameter depends on $N$ and satisfies
\[
\kappa=\kappa_N=c\sqrt N,
\qquad c>0 \ \text{fixed}.
\]

\subsection{Statement of the main result}

In contrast with the previous regime (where $\kappa$ is fixed), we will prove that as $N\uparrow \infty$ the subtree $\Upsilon_L$ connecting $L$ to $\Delta$ is now in general a bouquet of several disjoint binary trees rooted at
 $\Delta$. We say that a $L$-tree is a bouquet of $r$ binary trees iff $L^*=L$, every internal vertex has exactly two children and the root exactly $r$.
 
\medskip
\noindent
Let $\pi=\{B_1,\dots,B_r\}$ be a partition of $L=\{1,\dots,l\}$ into $r$
non–empty blocks, with $|B_j|=l_j$. The configurations  we consider
are the bouquets  of $r$ binary $B_j$-trees rooted at $\Delta$. Their set of leaves is exactly $L$.
For each block $B_j$ (with $|B_j|=l_j$), let $\alpha_j\in\mathcal B_{l_j}$
be a binary $B_j$–tree class.
 The corresponding $Q_L$ can then be identified with the $r$-tuple $(\alpha_j,\; j=1,...r)$.  Recall that we denote by \(\Pi_N\) the partition of 
\(L\) induced by $Q_L$.\\

\medskip
\noindent
As in the case $r=1$, to fully determine the class of $\Upsilon_L$, the $Q_L$'s are completed by $\mathbb N$–extensions:
As before, we denote by  $u=(u_i)_{1\le i\le 2l-r}$  
the numbers of unlabelled vertices lying on the corresponding reduced-edge
segments
and set $\Sigma=\sum_{i=1}^{2l-r} u_i$.

\medskip
\noindent
By Lemma~\ref{lemdet2}, the probability of a given configuration
$(\pi,(\alpha_j)_{1\le j\le r},u)$ is
\begin{equation}\label{eq:forest-weight}
\mathbb P^{(N),\kappa}
(\pi,(\alpha_j),u)
=
\frac{\kappa^{r-1}(l+\Sigma+\kappa)}
     {(N+\kappa)^{l+\Sigma}}
\prod_{i=0}^{\Sigma-1}(N-l-i).
\end{equation}
An easy recursion shows that 
\[
C_{l,r}
=
\sum_{\pi\in\mathcal P_{l,r}}
\prod_{B\in\pi} c_{|B|}
\]
is the number of classes of bouquets of $r$ binary trees whose leaves are $L$
( with $\mathcal P_{l,r}$ denoting the set of partitions of $L$ into $r$ blocks).

 \begin{theorem}[Critical regime $\kappa=c\sqrt N$]\label{thm:critical}
Fix $l\ge 1$ and $c>0$, and set $\kappa=\kappa_N=c\sqrt N$.
Let $L=\{1,\dots,l\}$.

\smallskip
\noindent
(a) \emph{Binary reduction.}
The probability that $\Upsilon_L$ is not a binary bouquet rooted at $\Delta$
tends to $0$ as $N\uparrow\infty$.

\smallskip
\noindent
(b) \emph{Limit law of reduced configurations.}
Let $\pi=\{B_1,\dots,B_r\}$ be a partition of $L$ into $r$ blocks and let
$\alpha_j\in\mathcal B_{|B_j|}$ be reduced binary tree classes on each block, as above.
Then
\[
\PP^{(N),c\sqrt N}\!\bigl(Q_L=(\pi,(\alpha_j)_{1\le j\le r})\bigr)
\longrightarrow
I_{l,r}(c),
\]
where  $I_{1,1}=1$ and for $l\ge 2$ $I_{l,r}(c)$ is given by:
\begin{equation}\label{eq:Ilr-1d}
I_{l,r}(c)=\frac{c^{\,r-1}}{(2l-r-2)!}\int_0^\infty
s^{2l-r-2}\exp\!\Bigl(-\frac{s^2}{2}-cs\Bigr)\,ds.
\end{equation}
The limiting partition of $L$ has EPPF
\[
p(l_1,\dots,l_r)
=
I_{l,r}(c)\prod_{i=1}^r c_{l_i},
\]
which is the \(\tfrac12\)-stable Poisson--Kingman EPPF.\\
In particular, writing
\[
C_{l,r}
=
\sum_{\pi\in\mathcal P_{l,r}} \prod_1^rc_{l_i}
\]
one has
\[
\PP^{(N),c\sqrt N}(|\Pi_N|=r)\longrightarrow C_{l,r}\,I_{l,r}(c),
\]
and conditionally on $|\Pi_N|=r$, the reduced configuration
$(\pi,(\alpha_j))$ is asymptotically uniform over the $C_{l,r}$ possibilities.

\smallskip
\noindent
(c) \emph{Limit for the extension vector.}
Let $(\pi,(\alpha_j))$ have $r$ blocks
For $t\in (0,\infty)^{2l-r}$, 
set $\sigma=\sum_{i=1}^{2l-r} t_i$.
Then
\[
\lim_{N\to\infty}
N^{(2l-r)/2}\,
\PP^{(N),c\sqrt N}\!\bigl(Q_L=(\pi,(\alpha_j)),\ u_L=[t\sqrt N]\bigr)
=
c^{\,r-1}(\sigma+c)\,
\exp\!\Bigl(-\frac{\sigma^2}{2}-c\sigma\Bigr).
\]
\end{theorem}
 
 \subsection{Proof}
 
 As in theorem \ref{blu}, the proof of (b) will be obtained from (c), 
with a Riemann-sum argument and a normalization identity. Once (b)
is known, (a) follows by summing over all bouquet configurations. We therefore focus on proving c).
 \noindent
\paragraph{Proof of c).}
Set $d=2l-r$ and let $t\in (0,\infty)^d$. Set $u_i=[t_i\sqrt N]$ and denote $\sigma=\sum_{i=1}^d t_i$ so that
$\Sigma=\sigma\sqrt N-\theta_N$, with $0\leq \theta_N\le d$.
Proceeding exactly as in the proof of part~c) of Theorem~1
(Taylor expansion of the logarithm and uniform error control on a  bounded
simplex), one obtains:

\begin{lemma}
For every $\sigma_{\max}>0$ there exist $C>0$  such that, uniformly for \(t\in\RR_+^d\) satisfying
\[
\sum_{i=1}^d t_i\le \sigma_{\max}:
\] and $N$ large enough,

\[
\Bigg|
N^{d/2}
\mathbb P^{(N),c\sqrt N}
\bigl(\pi,(\alpha_j),u=[t\sqrt N]\bigr)
-
c^{\,r-1}(\sigma+c)
e^{-\frac{\sigma^2}{2}-c\sigma}
\Bigg|
\le \frac{C}{\sqrt N}.
\]
\end{lemma}
\medskip
\noindent
In particular, for fixed $t$, as $N\uparrow\infty$:
\begin{equation}\label{eq:pointwise-limit}
N^{d/2}
\mathbb P^{(N),c\sqrt N}
\bigl(\pi,(\alpha_j),u=[t\sqrt N]\bigr)
\longrightarrow
c^{\,r-1}(\sigma+c)
e^{-\frac{\sigma^2}{2}-c\sigma}.
\end{equation}

\paragraph{Riemann sums and lower bound.}
Fix \(\sigma_{\max}>0\), and set
\[
\Delta_{\sigma_{\max}}^{(d)}
:=
\left\{
t\in\RR_+^d:\ \sum_{i=1}^d t_i\le \sigma_{\max}
\right\}.
\]
Summing the local estimate over all admissible
\(u\in\NN^d\) such that \(\sum_i u_i\le\sigma_{\max}\sqrt N\) gives, by the same Riemann-sum argument as in the fixed-$\kappa$ case
\[
\liminf_{N\to\infty}
\sum_{\substack{u\in\NN^d\\ \sum_i u_i\le \sigma_{\max}\sqrt N}}
\mathbb P^{(N),c\sqrt N}
(\pi,(\alpha_j),u)
\ge
\int_{\Delta_{\sigma_{\max}}^{(d)}}
c^{\,r-1}(\sigma+c)
e^{-\frac{\sigma^2}{2}-c\sigma}
\,dt,
\]
where
\[
\sigma=\sum_{i=1}^d t_i.
\]
Letting \(\sigma_{\max}\to\infty\) and using monotone convergence yields
\[
\liminf_{N\to\infty}
\mathbb P^{(N),c\sqrt N}
(\pi,(\alpha_j))
\ge
\int_{\RR_+^d}
c^{\,r-1}(\sigma+c)
e^{-\frac{\sigma^2}{2}-c\sigma}
\,dt .
\]

Since the integrand depends on \(t\) only through
\(\sigma=\sum_{i=1}^d t_i\), the simplex-volume identity gives
\[
\int_{\RR_+^d}
c^{\,r-1}(\sigma+c)
e^{-\frac{\sigma^2}{2}-c\sigma}
\,dt
=
\frac{c^{\,r-1}}{(d-1)!}
\int_0^\infty
(\sigma+c)\sigma^{d-1}
e^{-\frac{\sigma^2}{2}-c\sigma}
\,d\sigma.
\]
Since \(d=2l-r\), this is
\[
\frac{c^{\,r-1}}{(2l-r-1)!}
\int_0^\infty
(\sigma+c)\sigma^{2l-r-1}
e^{-\frac{\sigma^2}{2}-c\sigma}
\,d\sigma.
\]
For \(l\ge2\), an integration by parts gives the equivalent form
\[
I_{l,r}(c)
=
\frac{c^{\,r-1}}{(2l-r-2)!}
\int_0^\infty
\sigma^{2l-r-2}
\exp\!\left(-\frac{\sigma^2}{2}-c\sigma\right)\,d\sigma.
\]

%
%
In the base case $l=1$, $r=1$, and
\[
I_{1,1}(c)=\int_0^\infty (\sigma+c)e^{-\sigma^2/2-c\sigma}\,d\sigma
=
\Big[-e^{-\sigma^2/2-c\sigma}\Big]_{0}^{\infty}=1.
\]

\medskip
\noindent

Summing the previous inequality over all
partitions $\pi$ with $r$ blocks and all binary choices $(\alpha_j)$,
we obtain that the total limiting mass of configurations with $r$
components is larger or equal to:
\[
M_{l,r}(c)
=
C_{l,r}\, I_{l,r}(c).
\]

Summing over $r=1,\dots,l$,
we obtain
\[
\liminf_{N\to\infty}
\sum_{r=1}^l
\mathbb P^{(N),c\sqrt N}
(\text{$r$ components})
\ \ge\
\sum_{r=1}^l C_{l,r}\,I_{l,r}(c).
\]
We have:
\begin{lemma}\label{zoz}
 $\sum_{r=1}^l C_{l,r}\,I_{l,r}(c)=1.$
\end{lemma}
It stems from the fact that $\prod_{i=1}^rc_{l_i }I_{l,r}(c)$ is an exchangeable partition probability function known as the $\frac 1 2$- stable Poisson-Kingman EPPF (see \cite{Pitman2003PK}, \cite{Pitman2006}, \cite{GnedinPitman2006}). For self-containedness, we present a short elementary proof below.\\
Since probabilities cannot sum to more than $1$,
it follows that all the above inequalities are in fact equalities.
In particular that:
\[
\lim_{N\to\infty}
\mathbb P^{(N),c\sqrt N}
(\pi,(\alpha_j))
=I_{l,r}(c).
\]
Adding on all binary trees configurations yields the result on block sizes.
\qed

\paragraph{Proof of lemma \ref{zoz}}
First, we show that for \(l\ge 2\) and \(1\le r\le l\),
\begin{equation}\label{eq:C-recursion}
C_{l,r}=(2l-r-2)C_{l-1,r}+C_{l-1,r-1},
\end{equation}
Indeed, form a weighted partition of \(\{1,\dots,l\}\) by adjoining the new
leaf \(l\) to a partition of \(\{1,\dots,l-1\}\). If \(l\) is a singleton,
one starts from \(r-1\) blocks, giving \(C_{l-1,r-1}\), since \(c_1=1\).
Otherwise \(l\) is inserted into a block \(B\) of size \(m\); using
\(c_{m+1}=(2m-1)c_m\), this multiplies the weight by \(2m-1\). Hence, for
a partition \(\pi'\) with \(r\) blocks,
\[
\sum_{B\in\pi'}(2|B|-1)=2(l-1)-r=2l-r-2, \]
so the non-singleton contribution is \((2l-r-2)C_{l-1,r}\). Summing the two
contributions gives the recursion.

\noindent
Then define
\[
S_l(c):=\sum_{r=1}^l C_{l,r}\,I_{l,r}(c).
\]
\noindent
Since $C_{1,1}=1$, we obtain $S_1(c)=1$.
\smallskip
\noindent
Fix $l\ge 2$.
Using \eqref{eq:C-recursion} and reindexing, we obtain
\begin{equation}\label{eq:S-split}
S_l(c)
=
\sum_{r=1}^{l-1} C_{l-1,r}\Big((2l-r-2)\,I_{l,r}(c)+I_{l,r+1}(c)\Big).
\end{equation}
An integration by parts of $\int_0^\infty
s^{2l-r-2}\exp\!\Bigl(-\frac{s^2}{2}-cs\Bigr)\,ds$ proves the identity
\begin{equation}\label{eq:bracket-id}
(2l-r-2)\,I_{l,r}(c)+I_{l,r+1}(c)=I_{l-1,r}(c),
\qquad 1\le r\le l-1.
\end{equation}

\smallskip
\noindent
Plugging \eqref{eq:bracket-id} into \eqref{eq:S-split} yields
\[
S_l(c)=\sum_{r=1}^{l-1} C_{l-1,r}\,I_{l-1,r}(c)=S_{l-1}(c).
\]
Therefore $S_l(c)=S_1(c)=1$ for all $l\ge 1$.

\qed
\medskip
\color{black}

\section{Regimes away from the critical scale}
\label{sec:off-critical}
%
\subsection{Cutting and monotonicity on the complete graph}

\begin{proposition}[Cutting and effective killing]
\label{prop:cutting-effective-killing}
Fix \(N\) and \(\kappa>0\). Let \(T\) be a
\(\Delta\)-rooted spanning tree on \(K_N\cup\{\Delta\}\) sampled according to
\(\mathbb P^{(N),\kappa}\). Let \(p\in(0,1]\). Independently
for each internal edge of \(T\), retain it with probability \(p\), and cut it
with probability \(z=1-p\); when an edge \((x,p_T(x))\) is cut, replace it by the
cemetery edge \((x,\Delta)\). Let \(\widehat F\) denote the forest produced by the cutting procedure.
Then \(\widehat F\) has the \(\kappa'\)-biased rooted forest law $\mathbb P^{(N),\kappa'}$, where
\[
\kappa'=\frac{\kappa+Nz}{p}.
\]

\end{proposition}
\begin{proof}

Let \(F\) be a fixed rooted forest on \(K_N\), with rooted components
\[
(T_1,x_1),\dots,(T_r,x_r),
\]
where \(x_i\) is the root of \(T_i\).  
Write
\[
n_i:=|T_i|.
\]
Thus \(F\) has \(N-r\) internal edges.

We compute the probability that the cutting procedure produces this particular
rooted forest \(F\). This probability is obtained by summing over all initial
\(\Delta\)-rooted spanning trees \(T\) which can lead to \(F\).

First, all internal edges of \(F\) must be present in \(T\), and they must
survive the cuts. This gives the factor
\[
p^{N-r}.
\]

It remains to choose the edges of \(T\) which connect the rooted components of
\(F\) to one another and to \(\Delta\). Contract each component \(T_i\) to a
single vertex, still denoted \(T_i\). In the contracted picture we obtain a
directed spanning tree on
\[
\{T_1,\dots,T_r,\Delta\},
\]
oriented towards \(\Delta\).

Now suppose that, in this contracted tree, the outgoing edge of \(T_i\) goes
to \(T_j\), with \(j\ne i\). In the original tree this edge must start from
the root \(x_i\) of the component \(T_i\), but its endpoint can be any vertex
of \(T_j\). Hence there are \(n_j\) original edges realizing the same
contracted edge \(T_i\to T_j\). Each such edge is internal and must be cut, so
it contributes the factor \(z\). Thus a contracted edge
\[
T_i\to T_j
\]
has total weight
\[
z\,n_j.
\]

On the other hand, if the outgoing edge of \(T_i\) goes directly to
\(\Delta\), then in the original tree this is the cemetery edge
\[
(x_i,\Delta),
\]
which has conductance \(\kappa\), and no cutting factor. Thus a contracted
edge
\[
T_i\to\Delta
\]
has weight
\(
\kappa
\).
Therefore
\[
\mathbb P(\widehat F=F)
=
\frac{p^{N-r}}{Z_N(\kappa)}
\sum_{\tau}
\prod_{T_i\to\Delta\in\tau}\kappa
\prod_{T_i\to T_j\in\tau} z\,n_j,
\]
where the sum is over all directed spanning trees \(\tau\) on
\(\{T_1,\dots,T_r,\Delta\}\) oriented towards \(\Delta\), and
\[
Z_N(\kappa)=\kappa(N+\kappa)^{N-1}
\]
is the normalizing constant for the initial \(\kappa\)-biased spanning-tree
law.

We denote the sum over \(\tau\) by \(\Theta_r\). It is the directed
matrix-tree sum for the contracted network with edge weights
\[
T_i\to\Delta:\ \kappa,
\qquad
T_i\to T_j:\ z\,n_j \quad (i\ne j).
\]

By the directed matrix-tree theorem, \(\Theta_r\) is the determinant of the
Dirichlet Laplacian \(M\) on the contracted vertices
\(\{T_1,\dots,T_r\}\). This matrix is
\[
M_{ij}
=
\begin{cases}
\kappa+z(N-n_i), & i=j,\\
-z\,n_j, & i\ne j.
\end{cases}
\]

We now compute this determinant. Let
\[
w:=\kappa+Nz,
\qquad
v:=
\begin{pmatrix}
z n_1\\
\vdots\\
z n_r
\end{pmatrix},
\qquad
\mathbf 1:=
\begin{pmatrix}
1\\
\vdots\\
1
\end{pmatrix}.
\]
Let \(v^\top\) denote the transpose of \(v\). Then
\[
M=wI-\mathbf 1 v^\top .
\]
%
Using the matrix determinant lemma,
\[
\det(wI-\mathbf 1 v^\top)
=
\det(wI)\left(1-v^\top(wI)^{-1}\mathbf 1\right).
\]
Since
\[
\det(wI)=w^r,
\qquad
(wI)^{-1}=\frac1w I,
\qquad
v^\top\mathbf 1
=
z\sum_{i=1}^r n_i
=
zN,
\]
we get
\[
\det M
=
w^r\left(1-\frac{zN}{w}\right)
=
w^{r-1}(w-zN)
=
\kappa(\kappa+Nz)^{r-1}.
\]
Thus
\[
\Theta_r=\kappa(\kappa+Nz)^{r-1}.
\]

Consequently
\[
\mathbb P(\widehat F=F)
=
\frac{
p^{N-r}(\kappa+Nz)^{r-1}
}{
(N+\kappa)^{N-1}
}
=
\frac{\kappa'^{\,r-1}}
     {(N+\kappa')^{N-1}}.
\]
This is exactly the probability of the rooted forest \(F\) under
\(\mathbb P^{(N),\kappa'}\). 

\end{proof}

 For vertices \(x,y\in K_N\), let us write
\(x\sim y\) if they belong to the same component of the forest obtained from
\(\Upsilon\) after deleting the root \(\Delta\), and \(x\not\sim y\)
otherwise. 
\begin{corollary}[Monotonicity on the complete graph]
\label{cor:complete-graph-monotonicity}
Let \(0<\kappa<\kappa'\). There is a coupling of
\(F_\kappa\)
and
\(F_{\kappa'}\) such that
\[
F_{\kappa'}\subseteq F_\kappa
\]
as internal edge sets. Consequently, for any two vertices \(x,y\),
\[
\kappa\longmapsto \PP^{(N),\kappa}(x\sim y)
\;\text{is decreasing,}
\,\text{and}\;
\kappa\longmapsto \PP^{(N),\kappa}(x\not\sim y)
\;\text{is increasing.}
\]
More generally, the induced partition of any fixed marked set is
stochastically increasing in the refinement order as \(\kappa\) increases.
\end{corollary}

\begin{proof}
Choose
$
p=\frac{N+\kappa}{N+\kappa'}$, then 
$\frac{\kappa+Nz}{p}=\kappa'$.
By Proposition~\ref{prop:cutting-effective-killing}, the forest obtained from
a \(\kappa\)-biased forest by independently cutting each internal edge with
probability \(1-p\) has law \(\PP^{(N),\kappa'}\). Since the operation only
deletes internal edges, the stochastic domination follows.
\end{proof}
%

\begin{theorem}\label{thm:trivial-regimes}
Let \(L\) be fixed.
\begin{itemize}
\item[(a)] If \(\kappa_N=o(\sqrt N)\), then \(\Pi_N\) converges in probability
to the one-block partition of \(L\).
\item[(b)] If \(\kappa_N\gg\sqrt N\), then \(\Pi_N\) converges in probability
to the discrete partition of \(L\). Moreover, letting, for \(x\in L\), \(D_N(x)\) denote the graph distance from \(x\)
to \(\Delta\) in \(\Upsilon\) we have:
\[
\max_{x\in L}\frac{D_N(x)}{\sqrt N}\longrightarrow 0
\]
in probability.
\end{itemize}
\end{theorem}

\begin{proof}[Proof of Theorem~\ref{thm:trivial-regimes}]
First consider two marked vertices. By Theorem~\ref{thm:critical} with
\(l=2\), for fixed \(c>0\),
\[
\lim_{N\to\infty}\mathbb P^{(N),c\sqrt N}(1\not\sim 2)
=I_{2,2}(c)
=c\int_0^\infty e^{-s^2/2-cs}\,ds
\le c\sqrt{\frac\pi2},
\]
and
\[
\lim_{N\to\infty}\mathbb P^{(N),c\sqrt N}(1\sim 2)
=I_{2,1}(c)
=\int_0^\infty s e^{-s^2/2-cs}\,ds
\le \frac1{c^2}.
\]
If \(\kappa_N=o(\sqrt N)\), then for every fixed \(c>0\), eventually
\(\kappa_N\le c\sqrt N\). By corollary~\ref{cor:complete-graph-monotonicity},
\[
\limsup_{N\to\infty}\mathbb P^{(N),\kappa_N}(1\not\sim 2)
\le I_{2,2}(c).
\]
Letting \(c\downarrow0\) gives
\(\mathbb P^{(N),\kappa_N}(1\not\sim 2)\to0\). A union bound over pairs in
\(L\) then gives convergence to the one-block partition.

Similarly, if \(\kappa_N\gg\sqrt N\), then for every fixed \(c>0\), eventually
\(\kappa_N\ge c\sqrt N\). Hence
\[
\limsup_{N\to\infty}\mathbb P^{(N),\kappa_N}(1\sim 2)
\le I_{2,1}(c),
\]
and letting \(c\to\infty\) gives
\(\mathbb P^{(N),\kappa_N}(1\sim 2)\to0\). Another union bound gives convergence to the discrete partition.\\

\medskip
\noindent
It remains to control the distances in the regime \(\kappa_N\gg\sqrt N\).
By the one-point case of Lemma~\ref{lemdet2}, for \(1\le d\le N\),
\[
\mathbb P^{(N),\kappa}(D_N(1)=d)
=
\frac{d+\kappa}{(N+\kappa)^d}
\prod_{i=0}^{d-2}(N-1-i),
\]
with the usual convention that an empty product equals \(1\).
Define
\[
a_m(\kappa):=\prod_{i=0}^{m-1}\frac{N-1-i}{N+\kappa},
\qquad a_0(\kappa)=1.
\]

Then
\[
\PP^{(N),\kappa}(D_N(1)=d)=a_{d-1}(\kappa)-a_d(\kappa).\]
After telescoping,
\[
\mathbb P^{(N),\kappa}(D_N(1)\ge m)
=
a_{m-1}(\kappa)=\prod_{i=0}^{m-2}\frac{N-1-i}{N+\kappa},
\qquad 1\le m\le N.
\]
This tail is decreasing in \(\kappa\). Therefore, for fixed \(t>0\) and
\(c>0\), with \(m_N=\lceil t\sqrt N\rceil\),
\[
\limsup_{N\to\infty}
\mathbb P^{(N),\kappa_N}\!\left(D_N(1)\ge t\sqrt N\right)
\le
\lim_{N\to\infty}
\prod_{i=0}^{m_N-2}\frac{N-1-i}{N+c\sqrt N}
=
\exp\!\left(-\frac{t^2}{2}-ct\right).
\]
Letting \(c\to\infty\) gives \(D_N(1)/\sqrt N\to0\) in probability. By symmetry, the same estimate holds for every \(x\in L\). Since \(L\) is
fixed, a union bound gives
\[
\max_{x\in L}\frac{D_N(x)}{\sqrt N}\longrightarrow0
\]
in probability.
\end{proof}

\begin{remark}
The proof of the preceding corollary is special to \(K_N\), because it uses the explicit
cutting identity of Proposition~\ref{prop:cutting-effective-killing}. The
monotonicity statement itself, however, is much more general.

Let \(G=(V,E)\) be a finite network with conductances \(c_e>0\). Let
\(q=(q_x)_{x\in V}\) and \(q'=(q'_x)_{x\in V}\) be two killing functions such
that
\[
q'_x\ge q_x,\qquad x\in V.
\]
Let \(F_q\) and \(F_{q'}\) be the rooted forests obtained by adding a cemetery
vertex \(\Delta\), with cemetery conductances \(q_x\) and \(q'_x\),
respectively. 
The transfer-current theorem in the finite conductance setting
\cite[\S8.3]{stfl} (see also \cite{BurtonPemantle1993} ) gives that the internal edge set is determinantal. Denoting the
 transfer-current
kernel by $K_q$, it is easy to check that  \(Q'\ge Q\), implies that 
in the positive semidefinite order:
\[
K_{q'}\le K_q .
\]
By stochastic domination for determinantal measures
\cite{LyonsDPM}, the determinantal process with kernel \(K_{q'}\) is
stochastically dominated by the one with kernel \(K_q\). In particular, connection events
\(\{x\sim y\}\) become less likely, while separation events and refinement
events for marked partitions become more likely.

Note however that for a general conductance network, if one starts from the
\(q\)-rooted forest law and then cuts each internal tree edge independently,
the law of the resulting forest is usually not another rooted-forest law with
a modified killing function. 
\end{remark}
%

\section{Continuous-time cutting and the induced fragmentation}
\label{sec:cutting}

\paragraph{The cutting process}
In this section, we consider the continuous-time cutting process where each internal edge
is cut at rate \(N^{-1/2}\). Denoting the time parameter by $a$, one has
\[
p=p_N(a)=e^{-a/\sqrt N},
\qquad
z=z_N(a)=1-e^{-a/\sqrt N},
\]
Applying Proposition~\ref{prop:cutting-effective-killing} with
\(\kappa=\kappa_N\), \(p=p_N(a)\), and \(z=z_N(a)\), we get the effective
parameter
\[
\kappa_N(a)
=
\frac{\kappa_N+Nz_N(a)}{p_N(a)}
=
(\kappa_N+N)e^{a/\sqrt N}-N.
\]
If
$
\frac{\kappa_N}{\sqrt N}\longrightarrow c,
$
then, for every fixed \(a\ge0\),
\(
\frac{\kappa_N(a)}{\sqrt N}\longrightarrow c+a.
\)

The time reversal of the finite-\(N\) cutting process is a coalescence process. At a
time where the parameter is \(\kappa\), let \(F\) be a rooted forest with
components \(T_1,\dots,T_r\), rooted respectively at
\(x_1,\dots,x_r\). The reversed chain chooses a component \(T_i\), removes its
cemetery edge \((x_i,\Delta)\), and connects \(x_i\) to a vertex of another
component \(T_j\). More precisely, each possible edge from \(x_i\) to a vertex
of \(T_j\), \(j\ne i\), is added at rate
\(
\frac{1}{\kappa\sqrt N}.
\)
Thus the ordered merger \(T_i\to T_j\) has rate
\(
\frac{|T_j|}{\kappa\sqrt N},
\)
and the unordered pair \(\{T_i,T_j\}\) merges at total rate
\(
\frac{|T_i|+|T_j|}{\kappa\sqrt N}.
\)
This is the finite-\(N\) coalescent dual of the cutting dynamics.

\paragraph{The limiting metric-bouquet fragmentation.}

We now pass to the finite-dimensional marked object. Let \(\mathcal M_L\)
denote the set of metric binary bouquets with leaves \(L\): an element is a
reduced binary bouquet \(q\), together with a positive length \(t_e\) attached
to each reduced edge \(e\in E(q)\). We write
\[
r(q)=\text{number of components of }q,
\quad
\sigma(t):=\sum_{e\in E(q)}t_e .
\]
At parameter \(c>0\), the limiting law
\(\nu_c\) on \(\mathcal M_L\) is given by the critical theorem \ref{thm:trivial-regimes}.

We now describe the limiting cutting operation. Let
\[
X_a=(q,(t_e)_{e\in E(q)}).
\]
Cuts fall on each reduced edge \(e\) with intensity equal to length measure on
that edge. Equivalently, during an infinitesimal time interval \(da\), the
edge \(e\) is hit with probability \(t_e\,da+o(da)\), and conditionally on
being hit the cut position is uniform along the interval of length \(t_e\).

Let us describe the resulting reduced bouquet. If \(e\) is a root stem of one
component, and if the cut is at distance \(u\in(0,t_e)\) from \(\Delta\), then
the part below the cut is erased. Thus the same
component remains, but its root stem has length \(t_e-u\).

Now suppose that \(e\) is not a root stem. Orient \(e\) away from \(\Delta\),
and let \(v\) be its lower endpoint. Let \(e_0\) be the edge ending at \(v\)
from below, and let \(e_1\) be the other edge starting from \(v\), namely the
sibling of \(e\). Since the bouquet is binary, after cutting \(e\), the vertex
\(v\) has only one child. It is therefore the only
vertex which must be suppressed when reducing again. Thus the lower component
is obtained by replacing the two edges \(e_0\) and \(e_1\) by a single edge of
length
\[
t_{e_0}+t_{e_1},
\]
whereas the upper part of \(e\) becomes a new root stem of length
\[
t_e-u.
\]
All other edge lengths are unchanged. We denote the resulting metric bouquet
by
\[
\operatorname{Cut}_{e,u}(q,t).
\]

The generator of the limiting cutting process is therefore
\[
(\mathcal A\Phi)(q,t)
=
\sum_{e\in E(q)}
\int_0^{t_e}
\left[
\Phi\bigl(\operatorname{Cut}_{e,u}(q,t)\bigr)
-
\Phi(q,t)
\right]\,du .
\]

The fixed-time marginals are
\[
X_0\sim\nu_c
\quad\Longrightarrow\quad
X_a\sim\nu_{c+a}.
\]
Indeed, if we start from the
finite-\(N\) model with
\[
\frac{\kappa_N}{\sqrt N}\longrightarrow c
\]
and cut each internal edge at rate \(N^{-1/2}\). At time \(a\), an edge has
survived with probability \(p_N(a)=e^{-a/\sqrt N}\). By
Proposition~\ref{prop:cutting-effective-killing}, the resulting finite
forest has the same law as the biased forest with parameter
\[
\kappa_N(a)
=
\frac{\kappa_N+N(1-p_N(a))}{p_N(a)}
=
(\kappa_N+N)e^{a/\sqrt N}-N.
\]
Hence
\[
\frac{\kappa_N(a)}{\sqrt N}\longrightarrow c+a.
\]
Applying the critical theorem at the new parameter \(c+a\) gives the marginal
law \(\nu_{c+a}\).

This argument also checks the transformation of the length coordinates. In the
finite tree, a cut on a microscopic chain removes the portion below the cut;
if a binary branch point becomes unary, it is suppressed and the two adjacent
remaining chains are concatenated. After division by \(\sqrt N\) and passage to the limit, this is
exactly the operation \(\operatorname{Cut}_{e,u}\) described above.

Equivalently, the same marginal identity may be checked directly at the
limiting level by verifying, for suitable test functions \(\Phi\), that
\[
\frac{d}{dc}\int_{\mathcal M_L}\Phi\,d\nu_c
=
\int_{\mathcal M_L}\mathcal A\Phi\,d\nu_c .
\]
This is the forward equation associated with the generator \(\mathcal A\).
\paragraph{Time reversal on metric bouquets.}
Let \(y=(q,t)\in\mathcal M_L\) be a metric bouquet, and write
\[
\sigma=\sum_{e\in E(q)}t_e .
\]
We denote by \(\mathcal R(q)\) the set of root stems of \(q\), one for each
component of the bouquet. Each component \(B\) of \(q\) is viewed as a metric
tree, endowed with its one-dimensional length measure, denoted by
\(d\ell_B\). Thus integration over \(B\) means integration along its edges.

The time reversal at parameter \(c\) is obtained from the adjoint relation
\[
\nu_c(dx)\,\mathcal A(x,dy)
=
\nu_c(dy)\,\mathcal A_c^*(y,dx)
\]
for off-diagonal transitions. For a test function \(\Phi\), the
reversed generator is the sum
\[
\mathcal A_c^*\Phi
=
\mathcal A_c^{*,0}\Phi+\mathcal A_c^{*,1}\Phi ,
\]
where the two terms are as follows.

The first term keeps the number of components fixed and lengthens one root
stem. If \(s\in\mathcal R(q)\) and \(h>0\), let
\(\operatorname{Ext}_{s,h}(q,t)\) be the metric bouquet obtained by adding a
segment of length \(h\) below the root stem \(s\). Then
\[
(\mathcal A_c^{*,0}\Phi)(q,t)
=
\sum_{s\in\mathcal R(q)}
\int_0^\infty
R_c^{(0)}(\sigma,h)
\Bigl[
\Phi(\operatorname{Ext}_{s,h}(q,t))-\Phi(q,t)
\Bigr]\,dh,
\]
where
\[
R_c^{(0)}(\sigma,h)
=
\frac{\sigma+h+c}{\sigma+c}
\exp\!\left(-\sigma h-\frac{h^2}{2}-ch\right).
\]

The second term merges two components by grafting one onto the other. Let
\(A\) and \(B\) be two distinct components of \(q\). For a point
\(x\in B\), chosen with respect to length measure \(d\ell_B(x)\), and for
\(h>0\), let
\(\operatorname{Graft}_{A\to(B,x),h}(q,t)\) be the bouquet obtained by
removing the cemetery attachment of \(A\), attaching the root of \(A\) to
\(x\) by   the root stem of \(A\) to which is added a segment of length \(h\).
Then
\[
(\mathcal A_c^{*,1}\Phi)(q,t)
=
\sum_{\substack{A,B\in\operatorname{Comp}(q)\\ A\ne B}}
\int_B
\int_0^\infty
R_c^{(1)}(\sigma,h)
\Bigl[
\Phi(\operatorname{Graft}_{A\to(B,x),h}(q,t))-\Phi(q,t)
\Bigr]\,dh\,d\ell_B(x),
\]
where
\[
R_c^{(1)}(\sigma,h)
=
\frac1c\,
\frac{\sigma+h+c}{\sigma+c}
\exp\!\left(-\sigma h-\frac{h^2}{2}-ch\right).
\]
This formula is used for \(c>0\).
The factors above come directly from detailed balance. In both reversed moves,
the new state has total length \(\sigma+h\). The ratio of the length-density
parts of \(\nu_c\) is
\[
\frac{\sigma+h+c}{\sigma+c}
\exp\!\left(-\frac{(\sigma+h)^2-\sigma^2}{2}-ch\right)
=
\frac{\sigma+h+c}{\sigma+c}
\exp\!\left(-\sigma h-\frac{h^2}{2}-ch\right).
\]
For root-stem lengthening, the number of components is unchanged, so there is
no extra factor. For grafting, the reversed move decreases the number of
components by one, and the factor \(c^{r-2}/c^{r-1}=1/c\) appears. The forward
cutting rate along the newly added segment is Lebesgue length measure \(dh\),
which gives the displayed rate densities.

Thus, forward in \(c\), the process cuts metric bouquets; backward in \(c\),
the dual process is obtained by summing over all possible root-stem
lengthenings and all possible ordered graftings of one component onto another.
\section{Relation to previous work}
The closest connections are with some results involving Brownian motion:
As $l=|L|$ increases, our results in the fixed-$\kappa$ regime yield a consistent
family of distributions on $\mathbb R_+$-extensions of binary tree classes with
$l$ leaves. This coincides with the family of ``proper $k$-trees'' introduced by
Aldous in \cite[\S4.3]{A3}. It is shown in \cite{A3} that this family admits a
representation in terms of the Brownian continuum random tree (CRT). Other
constructions of the CRT are given in \cite{A1,A2} using Brownian excursions or
branching processes; see also \cite{NP1,NP2,LG}.

Pitman~\cite{Pitman2003PK} studied partition structures derived from the zero set of Brownian motion
and from stable subordinators, by sampling uniform times and grouping them by excursion intervals.
The limiting EPPF obtained in Section~3
coincides with the excursion partition induced by sampling $l$ uniform times
from a reflected Brownian bridge, conditionally on its local time. It was obtained by Pitman in \cite{Pitman2003PK}, section 8-2. See also \cite{Pitman2006}. This is clearly related to the correspondence between trees and Dyck paths.
The construction time cutting process defined in section 5 is the finite-dimensional marked version of a
fragmentation/coalescence picture naturally related to the
Aldous--Pitman Poisson cutting construction of the Brownian CRT, whose time
reversal gives the standard additive coalescent \cite{AldousPitman1998}, and
to Bertoin's general framework of exchangeable fragmentation processes
\cite{Bertoin2001,Bertoin2006}.\\

On the complete graph, after the parameter change \(\beta=\kappa^{-1}\),
the \(\kappa\)-biased forest measure considered here is the arboreal gas
 (the $q\to0$ random-cluster model); see,
for instance \cite{BCS}, \cite{RayXiao}.\\
 Luczak and Pittel \cite{LuczakPittel1992}
consider forests chosen uniformly among all forests with a prescribed number of edges, equivalently a prescribed number of components $r$, and derive detailed asymptotics for component sizes, in
particular in the regime $r\asymp \sqrt N$. Since our $\kappa$-biased model
assigns weight $\kappa^{\,r-1}$ to forests with $r$ components, it may be viewed
as a mixture over $r$ of the Luczak--Pittel uniform model. Here we do not
condition on the number of components and instead analyze the unconditioned
model in different scaling regimes of $\kappa$.\\
Martin and Yeo \cite{MartinYeo2018} study critical random forests arising from
Erd\H{o}s--R\'enyi random graphs conditioned to be acyclic, and obtain global
scaling limits for component sizes in the critical window. Pavlov
\cite{Pavlov2000} surveys earlier results on limit distributions for component
sizes, distances and extremal tree quantities in random forests, primarily in
the uniform setting.


\end{document}